\documentclass[11pt,a4paper,reqno]{amsart} 

\usepackage{tikz}
\usepackage{amsmath,bm,bbm,amsthm, amssymb}
\usepackage{fullpage}
\usepackage{color}
\usepackage{todonotes}
\usepackage{ulem}
\usetikzlibrary{arrows.meta}
\usetikzlibrary{calc}
\usetikzlibrary{patterns}
\usetikzlibrary{arrows}
\usetikzlibrary{decorations.pathreplacing}
\usepackage[utf8]{inputenc}
\usepackage{pgfplots}
\definecolor{wiasblue}   {cmyk}{1.0, 0.60, 0, 0}
\definecolor{mlugreen}{RGB}{0,81,51}

\def\Z{\mathbb{Z}}
\def\E{\mathbb{E}}
\def\P{\mathbb{P}}
\def\Q{\mathbb{Q}}

\def\mc{\mathcal}
\def\ms{\mathsf}

\renewcommand\ge{\geqslant}
\renewcommand\le{\leqslant}

\def\e{\varepsilon}
\def\z{\zeta}
\def\th{\theta}
\def\la{\lambda}
\def\t{\tau}
\def\f{\frac}
\def\g{\gamma}
\def\de{\delta}
\def\d{{\rm d}}
\def\r{\rho}

\def\one{\mathbbmss{1}}
\def\ms{\mathsf}
\def\mc{\mathcal}

\def\Ip{\mc I^+}
\def\Ipm{\mc I^\pm}
\def\Im{\mc I^-}

\def\a{\alpha}
\def\b{\beta}
\def\L{\Lambda}
\def\Efr{E_{\ms b}}
\def\Era{E_{\ms c}}
\def\La{\Lambda}
\def\i{\infty}
\def\lf{\lfloor}
\def\rf{\rfloor}
\def\tf{\tfrac}
\def\tc{T^\b}

\definecolor{halfgray}{gray}{0.55}
\definecolor{webgreen}{rgb}{0,.5,0}
\definecolor{webbrown}{rgb}{.6,0,0}
\definecolor{Maroon}{cmyk}{0, 0.87, 0.68, 0.32}
\definecolor{RoyalBlue}{cmyk}{1, 0.50, 0, 0}
\definecolor{Black}{cmyk}{0, 0, 0, 0}
\definecolor{pinkish}{RGB}{255, 192, 203}
\definecolor{orange}{rgb}{216, 64, 0}

\newcommand{\CH}[1]{}%\todo[color=pinkish, inline]{\color{black} CH: #1}}

\newcommand{\HP}[1]{}%\todo[color=cyan, inline]{\color{black} HP: #1}}

\def\Xs{X^{\ms s}}

\theoremstyle{plain}
\newtheorem{theorem}{Theorem}[section]
\newtheorem{proposition}[theorem]{Proposition}

\newtheorem{lemma}[theorem]{Lemma}

\theoremstyle{definition}

\newtheorem{example}[theorem]{Example}

\theoremstyle{remark}

\date{\today}

\begin{document}

\title{Invariance principle for random walks on dynamically averaging random conductances}
\author{Stein Andreas Bethuelsen}
\author{Christian Hirsch}
\author{Christian M\"onch}
\address[Stein Andread Bethuelsen]{University of Bergen, Department of Mathematics, Allegaten 41, 5020 Bergen, Norway}
\email{stein.bethuelsen@uib.no}
\address[Christian Hirsch]{University of Groningen, Bernoulli Institute, Nijenborgh 9, 9747 AG Groningen, The Netherlands.}
\email{c.p.hirsch@rug.nl}
\address[Christian M\"onch]{Johannes Gutenberg-Universit\"at Mainz, Institut f\"ur Mathematik,  Staudingerweg 9, 55099 Mainz, Germany.
}
\email{cmoench@uni-mainz.de}

\begin{abstract}
	We prove an invariance principle for continuous-time random walks in a dynamically averaging environment on $\mathbb Z$. In the beginning, the conductances may fluctuate substantially, but we assume that as time proceeds, the fluctuations decrease according to a typical diffusive scaling and eventually approach constant unit conductances. The proof relies on a coupling with the standard continuous time simple random walk.
\end{abstract}

\maketitle
\section{Introduction}

%%% General motivation: Random conductance models

Since its inception in \cite{Solomon75}, random walks in random environments have evolved into a flourishing field of research with manifold connections to other branches in probability, physics and chemistry. Inside this all-encompassing framework, random walks on dynamic random conductances form a rich class of models. Here, the walker moves according to jump rates on the edges of the underlying graph that are evolving according to some stochastic process simultaneously with the movement of the walker. We refer the reader to \cite{rwdrewnum} for an excellent overview of the plethora of models studied in literature and to \cite{ACDS18, Biskup19} on the topic of invariance principles.

%ONLY STATIONARY SO FAR
Despite the variety of different models considered to date, they almost exclusively do not deviate from one central assumption: the environment should be \emph{time-stationary}. 
Results outside this scope are rare and often require rather strong mixing assumptions \cite{ABF18,denhollander2014,RedigVoell13} and prove weaker results (e.g.\ LLN) \cite{avena11,Bet16,BetHey17}.

%AVERAGING
However, the time-stationary setting completely ignores a very natural form of dynamics, namely those converging to a common deterministic limiting value with decaying fluctuations over time. For a simple example, we may think of the time-averages of renewal processes attached to each of the edges. Moreover, since the law of large numbers is so ubiquitous also far more elaborate models such as KPZ-type interface growth processes are of this nature \cite{kpz, rough}.

%TIME AVERAGING
We show that when considering  such an instationary time-averaging setting on the line $\Z$, then an invariance principle holds under surprisingly general conditions. Our central requirement is a sharp concentration of super-diffusive fluctuations for increments in the time-evolution of the environment. In particular, we do not need to put any kind of mixing condition -- be it in space or time. As a specific example, we show that environments based on renewal processes fit into this setting. 
The proof of the invariance principle relies on a coupling construction to a simple random walk crucially exploiting the one-dimensional structure of the underlying graph.

%OUTLINE
In Section \ref{mod_sec}, we introduce precisely the super-diffusive concentration condition and state the invariance principle, which is then proved in Section \ref{sec:proofs}.

\section{Model and invariance principle}
\label{mod_sec}

%
%ENVIRONMENT & WALKER
%
Consider the integer lattice $\Z = (V, E)$ with edges drawn between successive sites and let $\{\L_e(\cdot)\}_{e \in E}$ be a family of almost surely non-decreasing stochastic processes on $[0, \infty)$ governing the time-evolution of the random environment.  We henceforth write $\L_e([a, b]) := \L_e(b) - \L_e(a)$ for the increment of $\L_e$ over an interval $[a, b] \subset [0, \infty)$.

A nearest-neighbor random walk $\{X(t)\}_{t \ge 1}$ on $\Z$ starts at $X(1) = 0$. Given $\{\L_e\}_{e \in E}$, the walker sitting at $X(t) = v$ at time $t$ jumps along an incident edge $e$ at rate $\La_e(t)/t$. We call $X(t)$ a \emph{capricious random walk} (CRW) and illustrate its transition dynamics  in
 Figure \ref{model_fig}.

\begin{figure}[!htpb]
	\begin{center}
\begin{tikzpicture}
%Z
	\draw[dashed] (-1,0)--(11,0);

	%RECTANGLES
	\fill[black!20!white] (0, 0) rectangle (2, 3);
	\draw[thick] (0, 0) rectangle (2, 3);
	
	\fill[black!20!white] (2, 0) rectangle (4, 6);
	\draw[thick] (2, 0) rectangle (4, 6);
	
	\fill[black!20!white] (4, 0) rectangle (6, 2.4);
	\draw[thick] (4, 0) rectangle (6, 2.4);
	
	\fill[black!20!white] (6, 0) rectangle (8, 5);
	\draw[thick] (6, 0) rectangle (8, 5);
	
	\fill[black!20!white] (8, 0) rectangle (10, 3.2);
	\draw[thick] (8, 0) rectangle (10, 3.2);

	%ARROWS

	\draw[-{Triangle[width=10pt,length=10.5pt]}, line width=5pt, black](0,2) -- (0.95, 2);
	\draw[-{Triangle[width=8pt,length=8pt]}, line width=3pt, blue](0,2) -- (.9, 2);
	\draw[-{Triangle[width=10pt,length=10.5pt]}, line width=5pt, black](2,2) -- (1.05, 2);
	\draw[-{Triangle[width=8pt,length=8pt]}, line width=3pt, blue](2,2) -- (1.1, 2);

	\draw[-{Triangle[width=30pt,length=10.5pt]}, line width=12pt, black](2,2) -- (2.95, 2);
	\draw[-{Triangle[width=24pt,length=8pt]}, line width=10.5pt, blue](2,2) -- (2.9, 2);
	\draw[-{Triangle[width=30pt,length=10.5pt]}, line width=12pt, black](4,2) -- (3.05, 2);
	\draw[-{Triangle[width=24pt,length=8pt]}, line width=10.5pt, blue](4,2) -- (3.1, 2);

	\draw[-{Triangle[width=8pt,length=10.5pt]}, line width=5pt, black](4,2) -- (4.95, 2);
	\draw[-{Triangle[width=6pt,length=8pt]}, line width=3pt, blue](4,2) -- (4.9, 2);
	\draw[-{Triangle[width=8pt,length=10.5pt]}, line width=5pt, black](6,2) -- (5.05, 2);
	\draw[-{Triangle[width=6pt,length=8pt]}, line width=3pt, blue](6,2) -- (5.1, 2);

	\draw[-{Triangle[width=22pt,length=10.5pt]}, line width=12pt, black](6,2) -- (6.95, 2);
	\draw[-{Triangle[width=18pt,length=8pt]}, line width=10.5pt, blue](6,2) -- (6.9, 2);
	\draw[-{Triangle[width=22pt,length=10.5pt]}, line width=12pt, black](8,2) -- (7.05, 2);
	\draw[-{Triangle[width=18pt,length=8pt]}, line width=10.5pt, blue](8,2) -- (7.1, 2);

	\draw[-{Triangle[width=12pt,length=10.5pt]}, line width=5pt, black](8,2) -- (8.95, 2);
	\draw[-{Triangle[width=10pt,length=8pt]}, line width=3pt, blue](8,2) -- (8.9, 2);
	\draw[-{Triangle[width=12pt,length=10.5pt]}, line width=5pt, black](10,2) -- (9.05, 2);
	\draw[-{Triangle[width=10pt,length=8pt]}, line width=3pt, blue](10,2) -- (9.1, 2);

	%NODES
	\fill (0, 0) circle (3pt);
	\fill (2, 0) circle (3pt);
	\fill (4, 0) circle (3pt);
	\fill (6, 0) circle (3pt);
	\fill (8, 0) circle (3pt);
	\fill (10, 0) circle (3pt);

	%LABELS

	\coordinate[label=-90:{{$\Z$}}] (A) at (5,0);
	\draw[thick, ->] (-.5, 1)--(-.5, 5);
	\coordinate[label=180:{{$\f{\L_e(t)}t$}}] (A) at (-.5,3);

\end{tikzpicture}
	\end{center}
	\caption{Transition scheme of the CRW. Arrow thickness represents jump rate towards edge.}
	\label{model_fig}
\end{figure}
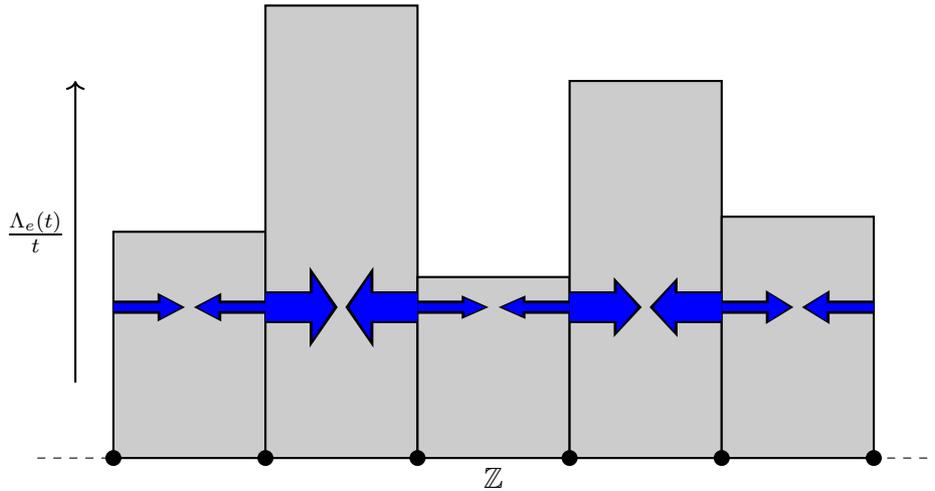

We assume that for some $\z > 5$ the following condition holds. For any $\e > 0$ and any interval $I \subset [0, \infty)$ of length $|I|$, 
\begin{equation}\label{eq:deviation}\tag{{\bf DIFF}}
	\sup_{e \in E}\P\big(\big|\L_e(I) - |I|\big| > |I|^{1/2 + \e}\big) \le c|I|^{-\z},
\end{equation}
where $c = c(\e)$ is a finite constant only depending on $\e$. In particular, an application of the Borel-Cantelli lemma yields that for every $e \in E$ almost surely $\L_e(t)/t\to 1$ as $t\to\infty$.

Before stating the main result, we illustrate that condition \eqref{eq:deviation} holds for environments induced by renewal processes.

%
%RENEWAL EXP
%
\begin{example}[Renewal process]
	Let $\{Y_n\}_{n \ge 1}$ be a sequence of non-atomic iid positive random variables with $\E[Y_1] = 1$ and having some finite exponential moment. Let $\{\L_e(t)\}_{t \ge 0}$ be the associated stationary renewal process \cite[Section 5.3]{asmussen}. Specifically, let $Y_0'$ be independent of $\{Y_n\}_{n \ge 1}$ and distributed according to the size-biased distribution of $Y_1$, and  set $Y_0 = UY_0'$ where $U \sim \ms{Unif}([0, 1])$ is an independent uniform random variable. Then, define the renewal process 
	$$\L_e(t) := \inf\big\{n \ge 0: \, S_n > t\big\},$$
	where $S_n := \sum_{0\le i \le n}Y_i$. Although $\{\L_e(t)\}_{t \ge 0}$ has stationary increments by \cite[Theorem 3.3]{asmussen}, the conductance $\L_e(t) / t$ is in general neither stationary nor Markovian. 

	Since the increments are stationary, it suffices to verify \eqref{eq:deviation} for intervals of the form $I = [0,t]$. To that end, put $n_-(t):= \lf t - t^{1/2 + \e}\rf$ and $n_+(t):= \lceil t + t^{1/2 + \e}\rceil$.
	Then, 
	$$\P(|\L_e(t) - t| \ge t^{1/2 + \e}) = \P(\L_e(t) \le n_-(t)) + \P(\L_e(t) \ge n_+(t)),$$
	and we explain how to deal with the first expression, noting that the arguments for the second are similar but easier. Now,
	$$\P(\L_e(t) \le n_-(t)) = \P(S_{n_-(t)} \ge t) \le \P(Y_0 \ge n_-(t)) + \P\big((S_{n_-(t)} - Y_0) - n_-(t) \ge t - 2n_-(t)\big).$$
	First, since the $Y_1$, and therefore $Y_0$, is assumed to have some exponential moment, the first expression decays stretched exponentially in $t$. Moreover, since $S_{n_-(t)} - Y_0$ is a sum of $n_-(t)$ iid mean-1 random variables, moderate deviations theory also implies that the second probability decays stretched exponentially in $t$. In particular, condition \eqref{eq:deviation} is satisfied for any $\z > 0$.
\end{example}

Although condition \eqref{eq:deviation} does not impose any constraints on mixing in space and time, it is stronger than it might appear at first sight as it concerns arbitrary intervals. For instance, thinking of first-passage percolation on $\Z^2$ and defining $\L_{\{i, i + 1\}}(t)$ as the first-passage time to a node $(t, i)$, condition \eqref{eq:deviation} holds for intervals of the form $I = [0, t]$. However, it seems questionable, whether this remains true for general intervals.

%
%INV PRINC
%
\begin{theorem}[Invariance principle]
	\label{mainThm}
	Assume  that \eqref{eq:deviation} holds with $\z > 5$. Then, for almost every realization $\omega$ of the environment $\{\L_e\}_{e \in E}$, the CRW in $\omega$ satisfies the invariance principle. That is, $\{X({tT})/\sqrt {2T}\}_{t \le 1}$  converges in distribution to standard Brownian motion.
\end{theorem}

We do not claim the concentration exponent $\z > 5$ to be optimal. However, we could well imagine that a pathologically slow concentration could lead to a significantly slowed down walker showing up as a time-scaled BM, or even lead to a transient walker.

The main tool to prove Theorem \ref{mainThm} is a finely adapted coupling with a continuous time symmetric simple random walk (SRW).

\section{Proofs}\label{sec:proofs}
%
%INIT
%
We now turn to the proof of Theorem~\ref{mainThm}. We first formulate three central auxiliary statements in Section~\ref{ssec:main} and show how they imply the invariance principle, then give detailed proofs of the auxiliary results in Sections~\ref{ssec:2}--\ref{init_sec}.
\subsection{Proof of Theorem \ref{mainThm}}\label{ssec:main}
We consider the total deviations of the CRW from the origin separately on different time scales. To make this precise, let 
$$R_{t_1, t_2} := \sup_{t_1 \le t \le t_2}|X(t) - X({t_1})|, \quad t_2 \ge t_1, $$
denote the \emph{range} of the walker in the time window $[t_1, t_2]$ and set $R_t := R_{0, t}$. Sometimes, in literature, the range is defined by taking the supremum of $|X(t) - X(t')|$ over pairs $t, t' \in [t_1, t_2]$, but the two variants differ only by a factor that is irrelevant for the following.

To lead up to the proof of Theorem \ref{mainThm}, we show that the range is essentially diffusive. This manifests itself in the following two propositions whose proofs are postponed to Sections \ref{ssec:2} and \ref{init_sec}, respectively.

%LIN DISP
\begin{proposition}[Linear displacement bound]
	\label{earlyProp}
It holds that \[\lim_{t \to \i}\P(R_t > 2t) = 0.\]
\end{proposition}

For $\{S(t)\}_{t \ge 0}$ a family of random variables and $f:[0, \i) \to [0, \i)$ any function, we write $S(t) \in o_p(f(t))$ if $S(t)/f(t)$ tends to 0 in probability as $t \to \i$.

%SMALL DISP
\begin{proposition}[Small displacement at intermediate times]
	\label{initProp}
	Let $(\z - 1)^{-1} < \a < \b \le 1$. Then, for every $\e > 0$, 
	$$R_{T^\a, T^\b} \in o_p( T^{\b - \a/2 + \e}).$$
	If, moreover, $\b < 1$, then 
$$	{R_{T^\a, T^\b}}\in o_p({\sqrt T}).$$
\end{proposition}

We infer Proposition~\ref{earlyProp} from a fairly rough estimate showing that the maximal number of jumps of $X$ grows at most linearly. The key ingredient to establish Proposition \ref{initProp} and to conclude the proof of Theorem \ref{mainThm} is a coupling of $\{X(t)\}_{t\le T}$ with a continuous-time symmetric random walk $\{\Xs(t)\}_{ t\le T}$ on $\Z$, that we discuss in more detail in Section~\ref{sec:coupling}. The coupling is initiated at time $\tc$, in position $\Xs(\tc) := X(\tc)$ for some fixed $\b < 1$ to be specified below and $\Xs$ jumps at a slightly diminished rate of $2(1 - \e_T)$, where $\e_T = \e_T(\g)$ decays as $T^{-\g}$, for some $\g < 1/2$.

%COUP PROP
\begin{proposition}[Coupling to SRW]
	\label{coupProp}
	Fix $\g < 1/2$ and let $\Xs$ denote a symmetric random walk jumping at rate $2(1- \e_T)$. There exists $ \b < 1$ such that $\{X(t)\}_{\tc \le t \le T}$ and $\{\Xs(t)\}_{\tc \le t \le T}$ can be coupled so that \[{\sup_{\tc \le t \le T} |X(t) - \Xs(t)|}\in o_p({\sqrt T }).\] 
\end{proposition}

%
%FROM DIFFERENCE TO INVARIANCE PRINCIPLE
%
The invariance principle is a consequence of Propositions~\ref{earlyProp} - \ref{coupProp}.
\begin{proof}[Proof of Theorem \ref{mainThm}]
	We decompose $X(t)$ as 
$ X(t) = (X(t) - Y(t)) + Y(t)
$,
	where for $t \le \tc$, we let $Y(t)$ be a rate-2 continuous time SRW independent of $X$ and $\Xs$, whereas for $t \ge \tc$, we set 
	$$Y(t) := \Xs\big((t - \tc)/(1 - \e_T) + \tc\big) + (Y(\tc) - \Xs(\tc)).$$
	In particular, $\{Y(t)\}_{t\le T}$ is a rate-$2$ continuous time SRW, so that we can apply the standard invariance principle to deduce that it converges to Brownian motion in distribution under diffusive rescaling.
To complete the argument, it thus suffices to show that 
	\[\sup_{t \le T}|X(t) - Y(t)| \in o_p(\sqrt T).\]

	To begin with, set $\a \in ((\z - 1)^{-1}, \f12)$ and note that ${\sup_{ t \le T^\a}|X(t)|} \in o_p(\sqrt T)$ by Propositions \ref{earlyProp}. Next, choose $\b$ as in Proposition \ref{coupProp}. Then, by Proposition \ref{initProp} also ${R_{T^\a, \tc}} \in o_p({\sqrt T})$. Now,
	$${\sup_{t \le \tc}|X(t)|} \le {2\sup_{t \le T^\a} |X(t)| + \sup_{T^\a \le t \le \tc}|X(t) - X({T^\a})|} = {2R_{T^\a} + R_{T^\a, \tc}},$$
	shows that $\sup_{t \le \tc}|X(t)| \in o_p(\sqrt T)$, and this is also the case when $X$ is replaced by $Y$. Therefore, it remains to control the deviation of $X(t)$ and $Y(t)$ for $t \ge \tc$. 

First, by Proposition~\ref{coupProp}, $\{T^{-1/2}(X(t) - \Xs(t))\}_{t\in [\tc, T]}$ vanishes in probability with respect to the sup-norm. Hence, it remains to show that 
	$$\sup_{\tc \le t \le T}\big|\Xs(t) - Y'(t)\big|\in o_p(\sqrt T),$$
	with $Y'(t) = \Xs\big((t - \tc)/(1 - \e_T) + \tc\big)$.
	To this end, we discretize and use the Markov property to obtain that
	\begin{align*}
		\P\Big(\hspace{-.1cm} \sup_{\tc \le t \le T}\hspace{-.1cm}|\Xs(t) - Y'(t)| > \de \sqrt T\Big) &\le \sum_{ i \le T}2\P\Big(2\sup_{t\le 2T\e_T}\hspace{-.1cm}| \Xs(t + \tc + i) - \Xs(\tc + i)| > \de \sqrt T\Big)\\
		&\le 2T\P\Big(2\sup_{t \le 2T\e_T}| \Xs(t + \tc) - \Xs(\tc)| > \de \sqrt T\Big).
	\end{align*}
Now, it follows from Doob's $L^p$-inequality that for $p \ge 2$, 
	\begin{align*}\label{eq:parts}
\P\Big(\sup_{t \le 2T\e_T}| \Xs(t + \tc) - \Xs(\tc)| > \de \sqrt T\Big)
		& \le (2/\de)^pT^{-p/2}\E\big[\big| \Xs(2T\e_T + \tc) - \Xs(\tc)\big|^p\big]\
	 \end{align*}
	 which is of order $O(\e_T^{p/2})$, since for the SRW $\Xs(t + \tc) - \Xs(\tc)$ and any $p > 1$, $\E[|\Xs(t + \tc) - \Xs(\tc)|^p]\in O(t^{p/2})$. Hence, choosing $p$ sufficiently large so that $ \e_T^{p/2}\in o(T^{-1})$ concludes the proof.
\end{proof}
The remainder of the paper is devoted to the proof of Propositions \ref{earlyProp}--\ref{coupProp}.

%%%%%%%%%%%%%%%%
%RANGE & JUMPS
%%%%%%%%%%%%%%%%

\subsection{Proof of Proposition \ref{earlyProp}} \label{ssec:2}
The main idea for proving Proposition \ref{earlyProp} consists of two steps. First, in Lemma \ref{envGrowLem} we leverage condition \eqref{eq:deviation} to establish a linear growth of the environment. Then, we invoke a Poisson concentration result to deduce that in this environment, the walker can travel at most at linear speed.

%
%JUMP RATE
%
\begin{lemma}[Uniform boundedness of jump rates]
	\label{envGrowLem}
		Let $\e > 0$. Then,
		$$\sup_{e \in E}\P\big(\sup_{s\ge t}|{\L_e(s)} - s|s^{-\f12 - \e} > 1 \big) \in O(t^{-(\z - 1)}).$$
\end{lemma}
\begin{proof}
	Put $\de(s) = s^{-\f12 + \e}$. First, 
	\begin{align*}
		\P \big(\sup_{s \ge t} (\L_e(s) - s)s^{-1}\de(s)^{-1}  > 1 \big) & \le \sum_{j \ge t }\P\big( \sup_{s\in [j - 1, j)} (\L_e(s) - s)s^{-1}\de(s)^{-1}  > 1  \big) \\ & \le \sum_{j \ge t}\P\big( \L_e(j) - j > j\de(j - 1) - (1 + \de(j - 1)) \big)
	\end{align*}
We have $j\de(j - 1) - (1 + \de(j - 1)) > j^{\f{1 + \e}2}$ for $j$ sufficiently large. From \eqref{eq:deviation}, it thus follows, that $\P\big( \L_e(j) - j > j\de(j - 1) - (1 + \de(j - 1)) \big)\le cj^{-\z}$ for all sufficiently large $j$ and therefore we can find $T > 1$ such that, for all $t\ge T$, 
\[
	\P \big(\sup_{s \ge t} (\L_e(s) - s)s^{-1}\de(s)^{-1}  > 1 \big)\le c \sum_{j \ge t} j^{-\z}\in O( t^{-(\z - 1)}). 
\]
	We conclude the proof by noting that bounds on the lower deviation of $\L_e(s)$ can be derived in a similar manner.
\end{proof}

%
%POIS CONC
%
We recall from \cite[Lemma 1.2]{Penrose03} a standard result on concentration of Poissonian random variables for ease of reference. 
\begin{lemma}[Poisson concentration]\label{lem:PoissConc}
Let $Z$ be a Poisson random variable with parameter $\la > 0$. Then,
	for all $x \ge \textup e^2\la$,
 \[\P(Z\ge x)\le \textup e^{-\f{x}2\log (x/\la)}. \]
\end{lemma}

Now, we have collected all ingredients for the proof of Proposition \ref{earlyProp}.

%
%EARLY PRF
%
\begin{proof}[Proof of Proposition~\ref{earlyProp}]
	Let $\a \in (1/(\z - 1), 1/2)$. We consider the range on the time intervals $[0, t^\a]$ and $[t^\a, t]$ separately. For the early times let 
	$$A_t := \big\{\max_{x\colon |x| \le t}\L_{\{x, x + 1\}}({t^\a}) \le 2t^\a\big\}$$
	be the event that until time $t^\a$ all edges at distance at most $t$ from the origin have weight at most $2t^\a$. From assumption \eqref{eq:deviation}, it follows that the event $A^t$ occurs with high probability (whp). Moreover, on the event $A^t$, the range $R_{t^\a}$ is stochastically bounded by a Poisson process with intensity $2t^\a$ on $[0, t^\a]$, unless $X$ (and therefore also the dominating Poisson process) leaves the set $[ - t, t]$. Thus, by Lemma~\ref{lem:PoissConc},
	\[
		\P(R_{t^\a} > 3t^{2\a}| A^t)\le \exp\big(-\tf32t^{2\a}\log(\tf32)\big), 
	\]
	hence $R_{t^\a} \le 3t^\a$ whp.
	
	Turning to the times in the interval $[t^\a, t]$, we argue similarly: set now 
	$$B_t := \Big\{\sup_{\substack{t^\a \le s \le t\\ x\colon |x| \le 2t}}\L_{\{x, x + 1\}}(s)/s \le 4/3 \Big\}, $$
as the event that, for all edges at distance at most $2t$ from the origin and all times $s \in [t^\a, t]$, the normalized weight $\L_e(s)/s$ is bounded above by $4/3$. By Lemma~\ref{envGrowLem}, $B_t$ occurs whp. On the event $A_t\cap B_t$, the range $R_{t^\a, t}$ is bounded above by a Poisson process with intensity $4/3$ on the interval $[t^\a, t]$ at least unless the dominating Poisson process jumps above $2t$. Again, by Poisson concentration, we conclude that
	$R_{t^\a, t}\le \tf{3t}2 $ whp.
\end{proof}

\subsection{Coupling CRW and delayed SRW}
\label{sec:coupling}
To estimate the fluctuation of $X$ more accurately, we now introduce the coupling with a SRW $\Xs$ formally. This coupling relies on a \emph{coupling time} $\tc$ and a \emph{jump delay} $\e_T = T^{- \g}$ both depending on the time horizon $T$.
Later, we will choose $\b$ and $\g$ to be a bit smaller than 1 and $1/2$, respectively.

We construct the coupling appearing in Proposition~\ref{coupProp} and used throughout the remaining sections via a graphical representation. Let $\{A_t(v, e)\}_{t\ge \tc}$ denote a Poisson point process of arrows directed from each node $v \in V$ along an incident $e$ with intensity measure $ t^{-1}\L_e(t) \d t$. 
 From these arrows, we jointly construct walks $\{\hat X(t)\}_{t \ge \tc}$ and $\{\hat X^{\ms s}(t)\}_{t \ge \tc}$. We start at $\hat X({\tc}) = \hat X^{\ms s}({\tc}) = X(\tc)$ and let $\hat X(t)$ always follow the arrows. 
Let $\{\t_i\}_{ i \ge 1}, $ denote the jump times of $\hat X$, $e_i$ the corresponding edges traversed and construct the jumps of $\hat X^{\ms s}$ recursively as follows: given $\{\hat X^{\ms s}(s)\}_{ s \le \t_i}$ and $\t_{i + 1} = t$, sample a uniform random variable 
$$
	U_t\sim \ms{Unif}[0, t^{-1}\L_t(e_i)]$$
independently of the collection of arrows $\{A_t(\cdot, \cdot)\}_t$ and of $\hat X^{\ms s}(s), s\in [0, \t_i)$. If $U_t \le 1 - \e_T$, then $\hat X^{\ms s}(t)$ moves along the arrow. If $U_t > 1 - \e_T$, then $\hat X^{\ms s}(t)$ does not move. That is, whenever $\hat X$ encounters an arrow from $A_t(v, e)$, the walker $\hat X^{\ms s}$ decides independently whether mimicking the movement of $\hat X^{\ms s}(t)$ or staying put.

\begin{lemma}[Coupling lemma]\label{lem:coupling}
	Let $\hat X^{\ms s}, \hat X$ be defined as above. Then, $\{\hat X(t)\}_{t\ge \tc}$ has the same distribution as $\{X(t) \}_{t\ge \tc}$. Furthermore, if $\b > (2\g) \vee (\z - 1)^{-1}$, then, whp, $\{\hat X^{\ms s}(t)\}_{\tc\le t\le T}$ can be coupled perfectly to the symmetric random walk $\{X^{\ms s}(t)\}_{\tc \le t \le T}$ starting at $X^{\ms s}({\tc}) = X({\tc})$ and jumping at rate $2(1 - \e_T)$.
\end{lemma}

\begin{proof}
	Let $v_i$ denote the vertex reached by the $i$th jump of $\hat X$, where $v_0 = X({\tc})$. We say the coupling \emph{fails before time $T$}, if the event 
	\[F_T := \Big\{\min_{i: \tc\le \t_i\le T}\inf_{\substack{\t_i \le t \le \t_{i + 1}\\ e\sim v_i}} t^{-1}{\L_e(t)} + < 1 - \e_T \Big\}\] 
	occurs.
	The significance of $F_T$ can be seen by noting that the jumps of $\hat X$ can be identified as a local thinning of the edge arrow processes $\{A_t(v, e)\}_{t\ge \tc}$, so we need to avoid the situation in which the local intensity of the edge arrow processes is too low to sustain the thinning. Since an independent thinning of a Poisson process is again Poisson, the marginals of the coupled walks indeed have the desired distributions, i.e., $\{\hat X(t)\}_{t\ge t_c}{ = } \{X(t)\}_{t\ge t_c}$ and $\{\hat X^{\ms s}(t)\}_{t\ge t_c}{ = } \{\Xs(t)\}_{t\ge t_c}$ in distribution, as long as the coupling succeeds.

	It remains to show that the coupling does not fail whp before time $T$. Let $\hat R_T$ denote the range of $\hat X$ until time $T$. Then,
\begin{equation}\label{eq:bdF}
\P(F_T) \le \P(F_T\cap \{\hat R_T\le 2T\}) + \P(\hat R_T > 2T), 
\end{equation}
	and on $\{\hat R_T\le 2T\}$ at most $4T + 2$ edges are involved in $F_T$. Hence, by Lemma~\ref{envGrowLem}, 
\[
\P(F_T\cap \{\hat R_T\le 2T\})\le (4T + 2)T^{-\b(\z - 1)}, 
\]
	so that we infer from \eqref{eq:bdF}, the choice of $\b$ and Proposition~\ref{earlyProp} that $\lim_{T \to \i}\P(F_T) = 0$.

\end{proof}

We henceforth work only with the coupled walks and omit the notational reference to the coupling, i.e., we consider realizations of $X$ and $\Xs$ such that $(X, \Xs) = (\hat X, \hat{\Xs})$ in distribution on $[\tc, T]$.

\subsection{Proof of Proposition \ref{coupProp}}

%
%REP OF DIFFERENCE
%
Under the coupling between the CRW and delayed SRW introduced in Section \ref{sec:coupling}, we obtain a powerful interpretation of the deviation $ X(t) - \Xs(t)$. To make it precise, we let $\{\t_i\}_{i \ge 1}$ denote the jump times of the CRW after the coupling time $\tc$ of Lemma \ref{lem:coupling}. Furthermore, we decompose the index set of these jump times as 
$$\{1, 2, \dots \} = : \Im \cup \Ip, $$
with $\Ipm$ corresponding to jumps to the left and to the right, respectively. Let now $\Ipm(t) = \{i\in \Ipm:\t_i\in [\tc, t] \}$ denote the jumps up to time $t > \tc$. In particular, for $t\in [\tc, T]$
$$X(t) = \#\Ip(t) - \#\Im(t)$$ 
and the increments of $X(\cdot) - \Xs(\cdot)$ correspond to a thinning of the jumps at $\{\t_i\}_{i \ge 1}$. More precisely, the difference in position can change only  when $\Xs$ stays put at one of the jump times. That is, for $t\in [\tc, T]$,
\begin{equation}\label{eq:diff}
X(t) - \Xs(t) = \#\{i \in \Ip(t)\colon \Xs({\t_i-})= \Xs(\t_i) \} - \#\{j \in \Im(t)\colon  \Xs({\t_j-})=\Xs(\t_j) \}, \; 
\end{equation}

%
%MARGIN
%
To analyze \eqref{eq:diff}, we marginalize out the randomness of the SRW. Denoting the edge traversed at time $\t_i$ by $e_i$, 
let 	\begin{align}
		\label{pi_eq}
	P_i := \P\big(\Xs({\t_i-})= \Xs(\t_i) \, |\, X, \L \big) = 1 - \f{\t_i-\e_T\t_i}{\L_{e_i}(\t_i)} 
	\end{align}
be
the probability to stay put for the SRW conditioned on the jump information of the CRW and the environment.
\begin{proposition}[Marginalizing the SRW]
	\label{margeProp} Assume $\z>5$. Then there exists $\r < 1/2$ such that
	$$\sup_{t \le T}\Big|\sum_{i \in \Ip(t)}P_i - \sum_{j \in \Im(t)}P_j\Big| \in o_p(T^\r).$$
\end{proposition}
As the proof of Proposition~\ref{margeProp} is a little lengthy, we defer it to the next section. To leverage Proposition~\ref{margeProp}, we need another lemma giving a linear upper bound on the jump counts.
\begin{lemma}[Linearity of jump counts]
	\label{countLem}
	It holds that
	$$\lim_{T \to \i}\P(\t_{\lf 6T\rf} \ge T) = 1.$$
\end{lemma}
\begin{proof}
By Proposition~\ref{earlyProp}, whp, at most $4T + 2$ edges are involved in the evolution of $X$ up to time $t$. 
Consequently, whp, the number of jumps of $X$ on $[\tc, T]$ is dominated by a Poisson($5T$)-distributed random variable $Z_T$. It follows that
\[
\P(\t_{\lf 6T\rf} \ge T)\le \P(Z_T \ge 6T), 
\]
and the latter vanishes by Poisson concentration. 
\end{proof}
%
%COUP PROB PROOF
%
Together with Lemma~\ref{countLem}, Proposition \ref{margeProp} yields Proposition \ref{coupProp}.
\begin{proof}[Proof of Proposition \ref{coupProp}]
	Let $\xi_i := \one\{\Xs(\t_i) = \Xs({\t_i-})\}$, $i \ge 1$ be the indicator that the SRW does not follow the CRW.
	 By the observations leading up to Equation \eqref{eq:diff} and Proposition \ref{margeProp}, it suffices to show that 
	$$\lim_{T \to \i}\P\Big(\sup_{\tc \le t \le T}\Big|\sum_{i \in \mc I^\pm_t} (\xi_i - P_i)\Big| \ge T^\r\Big) = 0, $$
	for both $\Ip$ and $\Im$, where $\r$ is specified in Proposition \ref{margeProp} and chosen such that $\r > \tf12 - \b \g$. By symmetry, we can contend ourselves with the statement for $\Ip$. Conditionally on $X$ and $\L$, the indicators $\{\xi_i\}_{i \ge 1}$, are independent Bernoulli($P_i$) random variables. Writing $Y(t) = \sum_{i\in \Ip(s)}\xi_i$ and $\Q(\cdot) = \P(\cdot\,|\,X, \L)$ for the conditional distribution, a standard concentration inequality such as \cite[Theorem 3.2]{ChungLu06} yields that
	\begin{equation}\label{eq:conditionalbound}
		\Q\big(\big|Y(t) - \E_{\Q} [Y(t)]\big| > T^\r\big)\le 2\exp\Big({-\f{T^{2\r}}{\E_{\Q}[ Y(t)] + T^\r/2}}\Big), \quad \tc \le t \le T.
	\end{equation}
	Now, $\E_{\Q} [Y(s)] = \sum_{i\in \Ip(s)}P_i$ and, $P_i \le 2\e_T$
		for all $i$ with $\tc \le \t_i \le T$ uniformly with probability exceeding $1 - T^{-\b (\z - 1) + 1}$ , by Lemma~\ref{envGrowLem} together with Proposition~\ref{earlyProp} and a union bound.
		
		 By Lemma~\ref{countLem}, there are at most $6T$ jumps of $X$ on $[\tc, t]$ and hence 
		\[
		\sum_{i\in \Ip(t)}P_i\le 6T \cdot 2 \e_T = 12 T^{- \g + 1}
		\]
		whp. 
	 It now follows from \eqref{eq:conditionalbound} that
		\[
			\P\Big(\sup_{\tc\le t \le T}\Big|\sum_{i \in \mc I^\pm_t} (\xi_i - P_i)\Big| \ge T^\r\Big)\le 2T\exp\big({-\f{T^{2\r}}{8 T^{- \g + 1} + T^\r/2}}\big) + \P(\#\Ip(T) > 6T), 
		\]
		which vanishes as $T\to\i$.
\end{proof}
		
%
%MARGINALIZED
%

\subsection{Proof of Proposition \ref{margeProp}} \label{ssec:6}
Since the random walk cannot circumvent edges, the visits to an edge to the right of $X(T^\b)$ occur in pairs of a left-to-right passage followed by a right-to-left passage (and the other way around for edges to the left of $X(T^\b)$, respectively). We formalize this observation by introducing a collection of pairings $\Pi_e \subset \Ip(T) \times \Im(T)$, $e\in E$, where  
$(i, j) \in \Pi_e$ if $e_i = e$ and 
$$j = \inf\{k \ge i\colon e_k = e, \t_k\le T\}$$
is the first index after $i$ where the edge $e$ is revisited before time $t$. It may happen, for each edge $e$, that at any time $t \in [\tc, T]$ at most two indices $j$ with $e_j = e$ stay unpaired. 
Hence, we can decompose the difference of the jump probabilities according to the visited edges:
\begin{align}
	\label{dec_diff_eq}
	\Big|	\sum_{i \in \Ip(t)}P_i - \sum_{j \in \Im(t)}P_j\Big| & = \Big|\sum_e\Big(\sum_{\substack{i \in \Ip(t) \\ e_i = e}}P_i - \sum_{\substack{j \in \Im(t) \\ e_j = e}}P_j\Big)\Big| \le \max_{i \in \Ipm(T)}2|P_i| +  \sum_{e \in E}B_e, 
\end{align}
where 
$$B_e :=  \sum_{(i, j) \in \Pi_e}|P_i - P_j|$$
denotes an upper bound on the \emph{bias} for edge $e$ accumulated until time $t$. Moreover, fix $\th := \f58 - \tf1{2(\z - 1)}$ 
and decompose the edge set $E$ into the sets
\begin{align*}
\Efr &:= \{e \in E\colon \#\{i \in \Ipm(T)\colon e_i = e\} > T^{\th}\}\text{ and }\\
	\Era &:= \{e \in E\colon 1 \le \#\{i \in \Ipm(T)\colon e_i = e\} \le T^{\th}\}
\end{align*}
of \emph{busy} and \emph{calm edges} visited more than $ T^{\th}$, respectively at most $T^\th$ times.

To prove Proposition \ref{margeProp}, we establish upper bounds for the per-edge bias in $\Era$ and $\Efr$ separately. 

%
%BUSY LEM
%
\begin{lemma}[Bias at busy and calm edges]
	\label{busyLem}
	Let $\e>0$. Then
	\begin{enumerate}
	\item $\max_{i \in \Ip(T) \cup \Im(T)}P_i \in o_p\Big( T^{\tf58-\b + \e}\Big).$
	\item $\max_{e \in \Efr}B_e \in o_p( T^{1-\b+\e})$ 
	\item $\max_{e \in \Era}B_e  \in o_p\Big(T^{\frac78 + \frac1{2(\zeta-1)}-\b+\e}\Big).$
	\end{enumerate}
\end{lemma}

%
%PRF MARGE
%
Before establishing the lemma, we show how to conclude the proof of the proposition.
\begin{proof}[Proof of Proposition \ref{margeProp}]
	First, by part 1 of Lemma \ref{busyLem}, we write 
	\begin{align*}
		\P\Big(\sup_{t \le T}\Big|\sum_{i \in \Ip(t)}P_i - \sum_{j \in \Im(t)}P_j\Big| \ge T^\r\Big) &\le  \P\Big(\sup_{\tc \le t \le T}\sum_{e \in E}B_e \ge T^\r \Big) + o(1).
	\end{align*}
	Let 
	$$F_{ \ms b} := \big\{\max_{e \in \Efr}B_e \le T^{1-\b+\e}\big\}\quad\text{ and }\quad
	F_{\ms c} := \Big\{\max_{e \in \Era}B_e \le T^{\frac78 + \frac1{2(\zeta-1)}-\b + 2\e}\Big\}$$
denote the events from Lemma \ref{busyLem}.
	First, we deal with the busy edges.
	 By Proposition~\ref{earlyProp}, we may assume that $\Efr$ contains at most $6T / T^\th = 6T^{1 - \th}$ edges. Thus, on the event $F_{\ms b}$, 
	\begin{equation}\label{margeProp1}\sum_{e \in \Efr}B_e \le  6T^{1 - \th + 1-\b + 2\e}, \end{equation}
	and we note that $2-\theta-\b + 2\e < 1/2$ for $\b$ and $\e$ sufficiently close to 1 and 0, respectively.

		%CALM 
	The argument for the calm edges is entirely analogous, we only need to derive a sharper bound for the number of calm edges than $6T$. 
	By Proposition \ref{initProp} applied with $\a = \b$ and $\b = 1$, we may assume that there are at most $T^{1 - \b/2}$ calm edges. Now, we calculate, that on $F_{ \ms c}\cap \{R_{\tc, T}\le T^{1 - \b/2 + 2\e}/2\}$,
		\begin{equation}\label{margeProp2}\sum_{e \in \Era}B_e \le T^{1 - \tf\b2 +\tf78 +\tf1{2(\zeta-1)} - \beta + 2\e}=  T^{\tf{15}8 +\tf1{2(\z - 1)}- \tf{3\b}2 + 2\e}.
		\end{equation}
		Particularly, $\tf{15}8 +\tf1{2(\z - 1)}- \tf{3\b}2 + \e< 1/2$ for $\b$ and $\e$ sufficiently close to 1 and 0, respectively. From the bounds in \eqref{margeProp1} and \eqref{margeProp2} we thus conclude the proof.
\end{proof}

%
%BUSY LEM PRF
%
 For the proof of Lemma~\ref{busyLem} we recall from \eqref{pi_eq} that
	$$P_i = \f{\L_{e_i}(\t_i) - \t_i + \t_i\e_T}{\L_{e_i}(\t_i)}.$$
\begin{proof}[Proof of Lemma \ref{busyLem}]%\phantom{a}\\[1ex]
	For all three parts, we rely heavily on Lemma \ref{envGrowLem}, which implies in particular that whp
	\begin{align}
		\label{busy_eq}
		{} \sup_{\substack{\tc \le s \le T\\ e \in \Efr \cup \Era }}|\L_e(s) -s | \le T^{5/8}. 
	\end{align}
In particular, $\inf\L_e(s)/s \ge 1/2$ for all $\tc \le s \le T$ and $e \in \Efr \cup \Era$.
	\\
\noindent	{\bf Part 1.}
First, we deduce from \eqref{busy_eq} that 
	$$P_i \le \f{T^{5/8} + T\e_T}{T^\b/2} \le 4T^{5/8 - \b},$$
	provided that $\g \ge 3/8$. \\[1ex]
	{\bf Part 2.}
 First, note that for $(i, j)\in \Pi_e$, 
	\begin{align*}
	P_i - P_j & = (1 - \e_T)\f{(\t_j - \t_i) \L_e(\t_j) - \L_e([\t_i, \t_j])\t_j}{\L_e(\t_i)\L_e(\t_j)}.
	\end{align*}
	Hence, by \eqref{busy_eq},
	$|P_i - P_j| \le 2T^{-\b} \big( (\t_j - \t_i) + \L_e([\t_i, \t_j]) \big).$
	Now, we conclude the proof by summing over $(i, j) \in \Pi_e$ and noting that $\L_e(T) \le 2T$ whp.\\[1ex]
	{\bf Part 3.}
The analysis is more delicate for rarely visited edges: there are many of them, so we have to show that the contribution of each individual edge vanishes.
	For a pair $(i, j) \in \Pi_e$, we consider the decomposition 
		\begin{align}
			\label{pairDecEq}
			(P_i - P_j)(1 - \e_T)^{-1} = \f{(\t_j-\t_i) (\L_e(\t_i) - \t_i)}{\L_e(\t_i)\L_e(\t_j)} - \f{\t_i(\L_e([\t_i, \t_j]) - (\t_j - \t_i))}{\L_e(\t_i)\L_e(\t_j)}, 
		\end{align}
		and bound the two summands separately. First, applying \eqref{busy_eq},
	$$ \f{(\t_j-\t_i) |\L_e(\t_i) - \t_i|}{\L_e(\t_i)\L_e(\t_j)} \le 4T^{5/8} 
		\f{\t_j - \t_i }{T^{2\b}},$$
		and summing over all pairs $(i, j) \in \Pi_e$ yields the bound $T^{13/8 - 2\b }$ which is smaller than $T^{-1/4}$ for $\b$ close to 1.
		
		It remains to bound the second term on the right-hand side in \eqref{pairDecEq}.  To bound 
		$$\sum_{(i, j) \in \Pi_e}\big|\L_e([\t_i, \t_j]) - (\t_j - \t_i)\big|, $$
	we fix $\a := \tf14 + \tf1{\z - 1} > \tf2{\z  - 1}$, and argue as in Lemma \ref{envGrowLem} to see that whp for $\e_0 := 1/16$,
	\begin{align}
		\label{pair_eq}
		\max_{e \in \Era}\sup_{\substack{ s, s' \le T\\ |s' - s|\ge T^\a}}|s - s'|^{-1/2 - \e_0}\big|\L_e([s, s']) -(s' - s)\big|.
	\end{align}

		Hence, by Jensen's inequality, 
		\begin{align*}
			\sum_{\substack{(i, j) \in \Pi_e\\ \t_j - \t_i \ge T^\a}}\hspace{-.1cm}\big|\L_e([\t_i, \t_j]) - (\t_j - \t_i)\big|\hspace{-.0cm} \le\hspace{-.2cm} \sum_{(i, j) \in \Pi_e}\hspace{-.2cm} (\t_j - \t_i)^{1/2 + \e_0} 
			\le (T^\th)^{1/2 - \e_0}T^{1/2 + \e_0}
			\le T^{(\th + 1)/2 + \e_0}.
		\end{align*}
 so that the contribution is at most $T^{(\th + 1)/2 - \b + \e} \le T^{-1/16}$.

	It remains to deal with the contributions from pairs satisfying $\t_j - \t_i \le T^\a$.
	First, under \eqref{pair_eq} we have $\L_e([\t_i, \t_j]) \le 2T^\a$, so that
	$$		\sum_{\substack{(i, j) \in \Pi_e\\ \t_j - \t_i \le T^\a}} \big|\L_e([\t_i, \t_j]) - (\t_j - \t_i)\big| \le \#\Pi_e T^\a \le T^{\th + \a}. $$
Inserting the definitions of $\th$ and $\a$ concludes the proof.
\end{proof}

%
%COUP CONSTR
%

\subsection{Proof of Proposition \ref{initProp}}
\label{init_sec}
The main idea to prove Proposition \ref{initProp} is to start by controlling deviations until time $T^\a$ for $\a < 1/2$ and then bootstrap to successively longer time-scales. To that end, we rely on a coupling with a SRW that is a small variant of the one introduced in Section \ref{sec:coupling}. Instead of starting the coupling at time $T^\b$, the coupling starts already earlier, namely at time $T^\a$. Moreover, the SRW is now slowed down stronger, namely by a factor $1 - \e_{T^\a} = 1 - T^{-\a\g}$ instead of $1 - T^{-\g}$. A fundamental consequence of this change is that in contrast to the scalings chosen in Section \ref{sec:coupling}, it is no longer necessarily the case that the SRW and the CRW deviate in at most $o_p(\sqrt T)$ many steps. Rather, we leverage the number of coupling failures together with the known range of the SRW in order to bound the range of the CRW.

\begin{proof}[Proof of Proposition \ref{initProp}]%\phantom a \\[1ex]
	{\bf Part 1.}
	By the law of the iterated logarithm, the displacement of the SRW until time $T^\b$ is at most of order $CT^{\b/2}\log\log t \le T^{\b - \a/2}$. Hence, it suffices to bound the deviation between the SRW and the CRW $\{X(t)\}_{t \le T^\b}$. Conditioning on $\{X(t)\}_{t \le T^\b}$ shows that there are whp at most $T^\b \e_{T^\a} = T^{\b - \g \a} $ coupling failures. 
	Hence, we obtain the desired worst-case deviation since $\g < 1/2$.\\[1ex]
	{\bf Part 2.}
	Recall that $0 < \a < \b < 1$ and by part 1, we may assume that $\a < 1/2$. We apply part 1 inductively to deviations over time scales $t_k = T^{1 - (1 - \a)^k}$ with growing $k$. For $k = 1$, this gives the time scale $T^\a$, where we can apply Proposition \ref{earlyProp}.
	
	Suppose that we have proven the desired result for the time scale $t_k$. 	Then, we apply part 1 with $\a = 1 - (1 - \a)^k$, $\b = 1 - (1 - \a)^{k + 1}$ for which
	\[ \b - \a/2 = 1 - (1 - \a)^{k + 1} - \f{1 - (1 - \a)^k}2 = 1/2 - \e(k) \]
	where $\e(k) = (1 - \a)^k(1/2 - \a) > 0$. By this, we conclude the proof since $t_k$ approaches 1 as $k\to \i$.
\end{proof}

\subsection*{Acknowledgement.} The authors thank H.~Pitters for inspiring discussions in the early phase of this project. We acknowledge support from the Deutsche Forschungsgemeinschaft (DFG, German Research Foundation) through ``Scientific Network Stochastic Processes on Evolving Networks''. The research of CM is funded by the DFG  through the priority programme SPP 2265 ``Random Geometric Structures''.

\bibliography{lit}

\begin{thebibliography}{10}

\bibitem{ACDS18}
S.~Andres, A.~Chiarini, J.-D. Deuschel, and M.~Slowik.
\newblock Quenched invariance principle for random walks with time-dependent
  ergodic degenerate weights.
\newblock {\em Ann. Probab.}, 46(1):302--336, 2018.

\bibitem{asmussen}
S.~Asmussen.
\newblock {\em Applied Probability and Queues}.
\newblock Springer, New York, second edition, 2003.

\bibitem{ABF18}
L.~Avena, O.~Blondel, and A.~Faggionato.
\newblock Analysis of random walks in dynamic random environments via
  {$L^2$}-perturbations.
\newblock {\em Stochastic Process. Appl.}, 128(10):3490--3530, 2018.

\bibitem{avena11}
L.~Avena, F.~den Hollander, and F.~Redig.
\newblock Law of large numbers for a class of random walks in dynamic random
  environments.
\newblock {\em Electron. J. Probab.}, 16:587--617, 2011.

\bibitem{Bet16}
S.~A. Bethuelsen.
\newblock The contact process as seen from a random walk.
\newblock {\em ALEA Lat. Am. J. Probab. Math. Stat.}, 15(1):571--585, 2018.

\bibitem{BetHey17}
S.~A. Bethuelsen and M.~Heydenreich.
\newblock Law of large numbers for random walks on attractive spin-flip
  dynamics.
\newblock {\em Stochastic Process. Appl.}, 127(7):2346 -- 2372, 2017.

\bibitem{Biskup19}
M.~Biskup.
\newblock An invariance principle for one-dimensional random walks among
  dynamical random conductances.
\newblock {\em Electron. J. Probab.}, 24:Paper No. 87, 29, 2019.

\bibitem{rwdrewnum}
O.~Blondel, M.~R. Hil{\'a}rio, and A.~Teixeira.
\newblock Random walks on dynamical random environments with nonuniform mixing.
\newblock {\em Ann. Probab.}, 48(4):2014--2051, 2020.

\bibitem{ChungLu06}
F.~Chung and L.~Lu.
\newblock Concentration inequalities and martingale inequalities.
\newblock {\em Internet Math.}, 3(1):79--127, 2006.

\bibitem{denhollander2014}
F.~den Hollander and R.~S. dos Santos.
\newblock Scaling of a random walk on a supercritical contact process.
\newblock {\em Ann. Inst. H. Poincar\'e Probab. Statist.}, 50(4):1276--1300,
  2014.

\bibitem{kpz}
M.~Kardar, G.~Parisi, and Y.-C. Zhang.
\newblock Dynamic scaling of growing interfaces.
\newblock {\em Phys. Rev. Lett.}, 56(9):889, 1986.

\bibitem{Penrose03}
M.~D. Penrose.
\newblock {\em Random Geometric Graphs}.
\newblock Oxford University Press, Oxford, 2003.

\bibitem{rough}
M.~D. Penrose.
\newblock Growth and roughness of the interface for ballistic deposition.
\newblock {\em J. Stat. Phys.}, 131(2):247--268, 2008.

\bibitem{RedigVoell13}
F.~Redig and F.~V\"{o}llering.
\newblock Random walks in dynamic random environments: a transference
  principle.
\newblock {\em Ann. Probab.}, 41(5):3157--3180, 2013.

\bibitem{Solomon75}
F.~Solomon.
\newblock Random walks in a random environment.
\newblock {\em Ann. Probability}, 3:1--31, 1975.

\end{thebibliography}
\bibliographystyle{abbrv}
\end{document}